\documentclass[11pt,reqno]{amsart}

\usepackage[utf8]{inputenc}
\usepackage{csquotes}
\usepackage[english]{babel}

\usepackage[toc, page]{appendix}

\usepackage[bookmarks=true]{hyperref}

\usepackage{amsmath}
\usepackage{amsthm}
\usepackage{amsfonts}
\usepackage{amssymb}
\usepackage{latexsym}
\usepackage{tikz}
\usetikzlibrary{calc, decorations.pathreplacing}
\usepgflibrary{decorations.markings}
\usepackage{subcaption}
\usepackage{dsfont}
\usepackage{array}
\usepackage{bm}
\usepackage{xcolor}
\usepackage{lmodern}
\usepackage[left=3cm,right=3cm,top=3cm,bottom=4cm]{geometry}

\usepackage{mdframed}

\usepackage{enumitem}

\numberwithin{equation}{section}
\newtheorem{thm}{Theorem}[section]
\newtheorem{cor}[thm]{Corollary}

\newtheorem{prop}[thm]{Proposition}

\theoremstyle{definition}
\newtheorem{definition}[thm]{Definition}
\newtheorem{hyp}[thm]{Assumption}
\theoremstyle{remark}
\newtheorem{rmq}[thm]{Remark}

\def \Var {\textrm{Var}}

\def \Cov {\textrm{Cov}}

\newcommand{\N}{\mathbb{N}}
\newcommand{\R}{\mathbb{R}}
\newcommand{\1}{\mathds{1}}
\newcommand{\Proba}{\mathbb{P}} 
\newcommand{\E}{\mathbb{E}} 
\newcommand{\Exp}{\mathcal{E}} 
\newcommand{\cvg}[3][ ]{ \underset{#2 \to #3}{\overset{#1}{\longrightarrow}}}

\DeclareMathOperator{\e}{e}

\newcommand{\mlaw}{\bold{m}}

\newcommand{\tpos}{t \geq 0}

\begin{document}
\date{\today}

\title[Limit theorem for Hawkes processes]{Limit theorems for Hawkes processes including inhibition.}

 \author[P. Cattiaux]{\textbf{\quad {Patrick} Cattiaux  \, }}
\address{{\bf {Patrick} CATTIAUX},\\ Institut de Math\'ematiques de Toulouse. CNRS UMR 5219. \\
Universit\'e Paul Sabatier,
\\ 118 route
de Narbonne, F-31062 Toulouse cedex 09.} \email{patrick.cattiaux@math.univ-toulouse.fr}

 \author[L. Colombani]{\textbf{\quad {Laetitia} Colombani  \, }}
\address{{\bf {Laetitia} COLOMBANI},\\ Institut de Math\'ematiques de Toulouse. CNRS UMR 5219. \\
Universit\'e Paul Sabatier,
\\ 118 route
de Narbonne, F-31062 Toulouse cedex 09.} \email{laetitia.colombani@math.univ-toulouse.fr}

 \author[M. Costa]{\textbf{\quad {Manon} Costa  \, }}
\address{{\bf {Manon} COSTA},\\ Institut de Math\'ematiques de Toulouse. CNRS UMR 5219. \\
Universit\'e Paul Sabatier,
\\ 118 route
de Narbonne, F-31062 Toulouse cedex 09.} \email{manon.costa@math.univ-toulouse.fr}

\maketitle

 \begin{center}

 \textsc{Universit\'e de Toulouse}
\smallskip

\end{center}

\begin{abstract}
In this paper we consider some non linear Hawkes processes with signed reproduction function (or memory kernel) thus exhibiting both self-excitation and inhibition. We provide a Law of Large Numbers, a Central Limit Theorem and large deviation results, as time growths to infinity. The proofs lie on a renewal structure for these processes introduced in \cite{costa} which leads to a comparison with cumulative processes. Explicit computations are made on some examples. Similar results have been obtained in the literature for self-exciting Hawkes processes only.
\end{abstract}
\bigskip


\bigskip

\bigskip

\section{Introduction.}\label{secintro}

Hawkes processes have been introduced by Hawkes \cite{hawkes71} and are widely used for modeling purposes: originally as models for the appearances of earthquakes \cite{hawkes71,hawkes74}, but now in finance \cite{hawkes18,bacry14} and econometrics or in neuroscience as models of spike trains of neurons \cite{evamas,RB}. We refer to the bibliography of our references for more details.
\medskip

A Hawkes process $t \mapsto N^h_t = N^h([0,t])$ is a point process on the real line $\mathbb R$ characterized by its initial condition on $]-\infty,0]$ and its intensity process $t \mapsto \Lambda(t)$ through the infinitesimal relation $$\mathbb P(N^h_. \textrm{ has a jump in } ]t,t+dt[|\mathcal F_t) \, = \, \Lambda(t) \, dt\,,$$ where $\mathcal F_t = \sigma (N^h(]-\infty,s[ \, ; \, s \leq t))$ is the natural filtration of the process and 
\begin{equation}\label{eqintens}
\Lambda(t) \, = \, f\left(\lambda \, + \, \int_{]-\infty,t[} \, h(t-s) \, N^h(ds)\right) \, .
\end{equation}
Here $\lambda \in \mathbb R$, $f: \mathbb R \to \mathbb R^+$ is the \emph{jump rate function} and $h: \mathbb R^+ \to \mathbb R$ is the \emph{reproduction function} (or \emph{memory kernel}). We shall give a more precise definition in the next section (in particular on what happens before time $0$) as well as results on existence and stability.
\smallskip

When $f$ is linear or affine, the process is said to be linear. In this case one has to assume that $\lambda \geq 0$ and $h \geq 0$ too. Note that when $h$ vanishes identically we recover a standard Poisson process. Otherwise the Hawkes process is called non linear. Actually, except for the  behaviour of the shifted process (\cite{Mas98,costa}), very few papers are dealing with possibly negative or signed $h$. The negative part of $h$ can be interpreted as self-inhibition.
\medskip

It is very natural to look at the large time behaviour of $N^h_.$, in particular the Law of Large Numbers (LLN) the Central Limit Theorem (CLT) and the deviations from the asymptotic mean or more generally the large deviations (LD). 

In the linear case (recall that $h$ is thus assumed to be non-negative) and assuming that $\parallel h\parallel_{L^1(du)} < 1$, both the LLN 
\begin{equation}\label{eq:lln_positive}\frac{N^h_t}{t} \, \to \, \frac{\lambda}{1-\parallel h\parallel_{L^1(du)}} := \, \mu \quad a.s. \textrm{ as } t \to +\infty \; ,
\end{equation}
and the CLT $$\frac{N^h_t \, - \, \mu t}{\sqrt t} \, \Longrightarrow \, \mathcal N^h(0,\sigma^2) \quad \textrm{ with } \sigma^2 \, = \, \frac{\lambda}{(1-\parallel h\parallel_{L^1(du)})^3} \; ,$$ where the convergence holds in distribution, have been shown (see e.g. \cite{DVJ}). Actually Bacry and al \cite{bacry13} have obtained the functional version of the CLT (convergence to some Brownian motion) in the multivariate case.  In a different direction, \cite{gaozhu} have shown a CLT for fixed $t$ as $\lambda \to +\infty$.

The easiest way to derive LLN and CLT in the linear case is presumably to use the immigration-birth representation also called the cluster process representation in \cite{hawkes74}, connecting $N^h$ to subcritical Galton-Watson processes. This representation was used in \cite{bordenave} in order to get the Large Deviation (LD) principle for $N^h_t/t$ with rate function $$I(x)=x \, \ln\left(\frac{x}{\lambda + x \, \parallel h\parallel_{L^1(du)}}\right) - x (1-\parallel h\parallel_{L^1(du)}) + \lambda \, .$$ For this explicit expression of the rate function see \cite{zhu13} p.761. The LD principle is obtained in \cite{bordenave} under the additional assumption $\int_0^{+\infty} t \, h(t) \, dt < +\infty$. It is claimed in the introduction of \cite{gaozhu21} that this assumption is not necessary. Under more restrictive assumptions, \cite{gaozhu21} contains precise deviations (see e.g. Theorem 2 therein).
\medskip

The non linear case is of course more difficult. According to the general seminal paper by Br\'emaud and Massouli\'e \cite{bre}, if $f$ is $L$-Lipschitz and $L \, \parallel h\parallel_{L^1(du)} \, < \, 1$, there exists a unique stationary version of the Hawkes process. Rate of convergence to equilibrium is studied in \cite{BNT} in two specific cases. As a consequence of Br\'emaud and Massouli\'e result, we get that 
\begin{equation}\label{eqergod}
\frac{N^h_t}{t} \, \to \, \mu = \mathbb E_s[N^h([0,1])]  \quad a.s. \textrm{ as } t \to +\infty \; ,
\end{equation} 
where $\mathbb E_s$ denotes the expectation w.r.t. the stationary ergodic distribution.

In the particular situation where $h$ is an exponential, the Hawkes process becomes Markovian and some results of large deviation have been obtained \cite{zhu15}. In \cite{zhu13}, Zhu proved a functional CLT \emph{at equilibrium} from which the following follows
\begin{thm}\label{thmzhu}
Assume that 
\begin{enumerate}
\item[(1)] \; $f$ is $L$-Lipschitz, 
\item[(2)] \; $h$ is non-negative, decreasing and such that $\int_0^{+\infty} t \, h(t) \, dt < +\infty$,
\item[(3)] \;  $L \, \int_0^{+\infty} \, h(t) \, dt \, < \, 1$,
\item[(4)] \; $\lambda \geq 0$.
\end{enumerate}
Then the stationary Hawkes process satisfies $\frac{N^h_t \, - \, \mu t}{\sqrt t} \, \Longrightarrow \, \mathcal N^h(0,\sigma^2)$ as $t \to +\infty$ in distribution, with $$\sigma^2 := \Var_s(N^h([0,1])) + 2 \sum_{j\geq 1} \, \Cov_s(N^h([0,1]),N^h([j,j+1]))$$ where $\Var_s$ and $\Cov_s$ denote the variance and covariance w.r.t. the stationary distribution.
\end{thm}
The proof is based on martingales techniques for the functional CLT. As the author himself is saying, to obtain an explicit expression for $\mu$ and $\sigma^2$ can rapidly become a difficult task. In the same work, Zhu also obtained a Strassen iterated logarithm law. One can also mention \cite{zhu14} where a large deviation result is obtained by contracting the level-3 LDP, i.e. by considering the shifted occupation measure. Theorem 2 in \cite{zhu14} then furnishes a LDP for $N^h_t/t$, provided $h$ is non decreasing and non-negative and $f$ is sub-linear at infinity. The expression of the rate function, as the infimum of the entropy on some set of measures satisfying a linear constraint is however not really tractable.
\medskip

Since we are interested in neurosciences, our goal in this work is to understand the role of self-inhibition in the asymptotic behaviour of Hawkes processes. Since inhibition will slow down the neuronal activity, we thus have to consider signed functions $h$ (the positive part modeling the self-excitation), but also jump rate functions $f$ satisfying $f(u)=0$ if $u\leq 0$. In the present paper, we will study the case of a general, signed, reproduction function with compact support and the specific jump rate function $f(u)=u^+= \max (u,0)$. This choice is of course the simplest one allowing us to introduce inhibition, and to compare this situation with linear models.
\medskip

We will obtain a LLN, a CLT and deviation inequalities, where the parameters are characterized by the renewal structure of the process introduced in \cite{costa} replacing the classical cluster representation of the self-exciting case established in \cite{hawkes74} which is no more valid. This renewal structure allows us to write the Hawkes process almost as a cumulative process.%

The main tools are then limit theorems for cumulative processes and actually, the technical work consists in showing that one can apply these theorems in the present situation. An important tool is a comparison between the considered Hawkes process, the self excited process associated to the positive part of the reproduction function, furnishing an upper bound, and a purely inhibited process corresponding to the (negative) lower bound of the reproduction function (see Proposition \ref{prop_minoration}), furnishing a lower bound.  

For simplicity we restrict ourselves to an empty initial condition (see below). Some explicit computations are done in simple particular cases of pure inhibition ($h$ non-positive). 
Precise statements will require some definitions, so that they are postponed to the next section. We emphasize, that the inhibition part introduces new intricacies. 

As we said, very few papers are dealing with inhibition. In \cite{evaesaim} some specific kernels are considered, but the addressed problem is not the one we are considering here. Looking at possibly negative reproduction functions is not only of mathematical interest. As shown in \cite{RB,RBfit,RBreconst} a multivalued version of the model we are studying is particularly well suited for modeling spike train of neurons, at least in an almost stationary regime. To extend our results to the multivalued framework should thus be an interesting question.
\medskip

\section{Notation, definitions and results.}\label{secresults}

\subsection{Hawkes processes}

We consider an appropriate filtered probability space $(\Omega, \mathcal{F}, (\mathcal{F}_t)_{t \geq 0}, \Proba)$ satisfying the usual assumptions. 
\begin{definition}
\label{def:Hawkes}
	Let $\lambda >0$ and $h: (0, +\infty) \rightarrow \R$ a signed measurable function. Let $N^0$ a locally finite point process on $(- \infty, 0]$ with law $\mlaw$. \\
	The point process $N^h$ on $\R$ is a Hawkes process on $(0, +\infty)$, with initial condition $N^0$ and reproduction measure $\mu(dt) = h(t) dt$ if: 
	\begin{itemize}
		\item $N^h\mid_{(-\infty, 0]} = N^0$,
		\item the conditional intensity measure of $N^h\mid_{(0, +\infty)}$ with respect to $(\mathcal{F}_t)_{t \geq 0}$ is absolutely continuous w.r.t the Lebesgue measure and has density:
		\begin{equation} \label{eq_intensite}
		\Lambda^h : t \in (0, +\infty) \mapsto \left( \lambda + \int_{(-\infty, t)} h(t-u) N^h(du) \right)^+ .
		\end{equation}
	\end{itemize}
where $x^+ = \max(x,0)$.
\end{definition}

The next proposition gives an explicit representation of the Hawkes process as solution of an SDE driven by a Poisson point process and states an important coupling property.
\begin{prop}[Proposition 2.1 in \cite{costa}]\label{prop_EDS}
	Let $Q$ be a $(\mathcal{F}_t)_{\tpos}$ - two-dimensional Poisson point process on $(0, + \infty) \times (0, +\infty)$ with unit intensity. We consider the equation
	\begin{align}\label{eq_EDS}
	\left\lbrace
	\begin{array}{l}
	N^h = N^0 + \int_{(0, +\infty) \times (0, +\infty)} \delta_u \1_{\theta \leq \Lambda^h(u)} Q(du, d\theta) \\
	\Lambda^h(u) = \left( \lambda + \int_{(-\infty, u)} h(u-s) N^h(ds) \right) ^+, ~ u>0, 	
	\end{array}
	\right.
	\end{align}
	where $\lambda >0$ is an immigration rate, $h:(0, +\infty)\rightarrow \R$ is a signed measurable function and $N^0$ is an initial condition of law $\mlaw$ on $(-\infty, 0]$.\\
	We consider the similar equation for $N^{h^+}$ in which $h$ is replaced by $h^+(.)=max(h(.),0)$. We assume that $\|h^+\|_1:=\parallel h^+\parallel_{L^1(du)} <1$ and that the distribution $\mlaw$ satisfies: 
	\begin{align}
	\forall t>0, \int_0^t \E_{\mlaw}\left( \int_{(-\infty, 0]} h^+(u-s)N^0(ds) \right) du < + \infty.
	\end{align}
	Then: 
	\begin{itemize}
		\item There exists a pathwise strong solution $N^h$ of equation \eqref{eq_EDS}, and this solution is a Hawkes process. 
		\item This property is true for $N^{h^+}$. Moreover, in the sense of measures, $N^h \leq N^{h^+}$, meaning that for all $0 \leq s \leq t <+\infty$, $N^h([s,t]) \leq N^{h^+}([s,t])$. 
	\end{itemize}
\end{prop}
\medskip

\subsection{Definitions and assumptions}
In this paper we consider a Hawkes process $N^h$ according to Definition \ref{def:Hawkes}. We focus on the case of a signed reproduction function $h$ which represents a possible inhibition on the appearance of future points. 
\begin{hyp}
\label{assumptions}

In all the paper, we will make the following assumptions : 
\begin{itemize}
		\item[$i)$]  $h: (0, +\infty) \rightarrow \R$ is a compactly supported signed measurable function. We define $L(h)$ as the supremum of the support of $h$: $L(h):= \sup\{ t>0, |h(t)|>0 \} <\infty$. 
		\item[$ii)$] $$\|h^+\|_1 := \int_0^{+\infty} \, h^+(u) \, du<1,$$
where $h^+(x)=max(h(x),0)$.
		\item[$iii)$] $\lambda >0$,
		\item[$iv)$] the initial condition on $]-\infty,0[$ does not contain any point i.e. $\mlaw=\delta_{\emptyset}$. 
	\end{itemize}
\end{hyp}

We are  interested in the asymptotic behaviour of the number of jumps of the process $N^h$ on the interval $[0,t]$, and we denote:
 $$N^h_t=N^h([0,t]), \quad  \forall t \geq 0$$
In particular we aim at quantifying precisely the loss of points due to inhibition. We will prove asymptotic results for $\frac{N^h_t}{t}$  and give  exact computations on specific examples.
\medskip

First we show another comparison result, this time furnishing a lower bound for $N^h_t$. This result motivates the detailed study of the canceling of intensity example.

\begin{prop}[Minoration of Hawkes process]\label{prop_minoration}~\\
Let $h$ be a function satisfying Assumptions \ref{assumptions}. Let $\lambda >0$ and define $g = - \lambda \1_{[0,L(h)]}$. \\
One can find a coupling of two Hawkes processes $N^h$ and $N^g$, respectively associated with the reproduction functions $h$ and $g$ and with basal intensity $\lambda$, such that for any $t\ge0$:
	$$N^h_t \geq N^g_t \quad ~ a.s. $$
\end{prop}
Note that this comparison result is weaker than the majoration via $h^+$, since we do not have $N^h([s,t])\geq N^g([s,t])$ for all $s$, but only for $s=0$.
\begin{proof}
The main idea is to construct these two processes with the same Poisson point process $Q$ on $(0, +\infty)^2$. We consider the successive jumps of $N^h$: $U^h_1, U^h_2, U^h_3, ...$; and the ones of $N^g$: $U^g_1, U^g_2, U^g_3, ...$. \\

We will prove by induction, that 
$$\forall j \geq 1, N^h_{U^g_j} \geq N^g_{U^g_j} = j \quad \text{a.s.}$$
 by studying the intervals associated with the  $[U^g_j, U^g_{j+1})$ for $j \in \N$.
 We stress out that considering the definition of the function $g$, the intensity $\Lambda^g$ of the Hawkes process $N^g$ can only take the two values $0$ and $\lambda$.

\noindent\underline{First interval:} First remark that $\forall t < \min(U^h_1, U^g_1)$,$$
\Lambda^h(t) = \lambda = \Lambda^g(t),$$
thus we have $U^h_1 = U^g_1$ and consequently $N^h_{U^g_1} = N^g_{U^g_1}$.

\noindent\underline{Second interval:} For $j=2$: by definition, there is only one jump for $N^g$ on $[U^g_1, U^g_2)$. There are two possibilities for $N^h$: 
		\begin{itemize}
			\item Assume that there is no other jump that $U^h_1$ in this interval.  Since $U^g_2 \geq U^g_1 + L(h)$, we have $\Lambda^g(U^g_2-) = \lambda = \Lambda^h(U^g_2-)$. Accordingly, $U^h_2 = U^g_2$ and in particular, $N^h_{U^g_2} = N^g_{U^g_2}$ a.s.
			\item Otherwise, there is at least one other jump of $N^h$ in $(U^g_1, U^g_1+L(h))$. In this case, $N^h_{U^g_2} \geq 2 = N^g_{U^g_2}$ a.s.
		\end{itemize}

\noindent\underline{Recursion step:} We fix $j$ and we suppose that the statement holds for $i \leq j$. Let $k = N^h_{U^g_j}\ge j$ by assumption. Then consider the two following cases:
\begin{itemize}
\item If $U^g_j$ is a jump of $N^h$, there is either at least one other jump of $N^h$ in $(U^g_j, U^g_{j+1})$ or no other jump. If there is no other jump, then  $\Lambda^h(U^g_{j+1}-) = \Lambda^g(U^g_{j+1}-) = \lambda$, since $U^g_{j+1} > U^g_j + L(h)$. So, $U^h_{k+1}=U^g_{j+1}$. In both situations, 
$$N^h_{U^g_{j+1}} \geq 1 + N^h_{U^g_{j}} \geq 1+ N^g_{U^g_j} = N^g_{U^g_{j+1}}.$$
\item If $U^g_j$ is not a jump of $N^h$, then $$\Lambda^h(U^g_j-) < \lambda = \Lambda^g(U^g_j-).$$
Therefore since the support of $h$ is of length $L(h)$ we deduce that 
$$U^h_k < U^g_j < U^h_k + L(h).$$
By the induction hypothesis, we know that $k \geq j$. Then, there is either at least one jump of $N^h$ in $(U^g_j, U^g_j+L(h))$, or the next jump is $U^g_{j+1}$, i.e. $U^h_{k+1} = U^g_{j+1}$. In both cases, we have $N^h_{U^g_{j+1}} \geq 1 + k \geq 1+ j = N^g_{U^g_{j+1}}$.
		\end{itemize}
		
This concludes the induction. Let us come back to a general $t \in \R^+$. For any fixed $\omega$, there exists $j=j(\omega) \in \N$, such that: $U^g_j(\omega) \leq t < U^g_{j+1}(\omega)$. Then using the monotonicity of $N^h$ we have
$$N^h_t \geq N^h_{U^g_j} \geq N^g_{U^g_j}  = N^g_t .$$
\end{proof}
Both comparison results may be used in the sequel.
\medskip

\subsection{Hawkes processes as cumulative processes.}
Our study of the large time behaviour of Hawkes processes lies on a renewal structure for Hawkes processes first introduced in \cite{costa} we shall partly recall below. Notice that this structure is used in \cite{costa} for a completely different purpose.

\begin{center}
	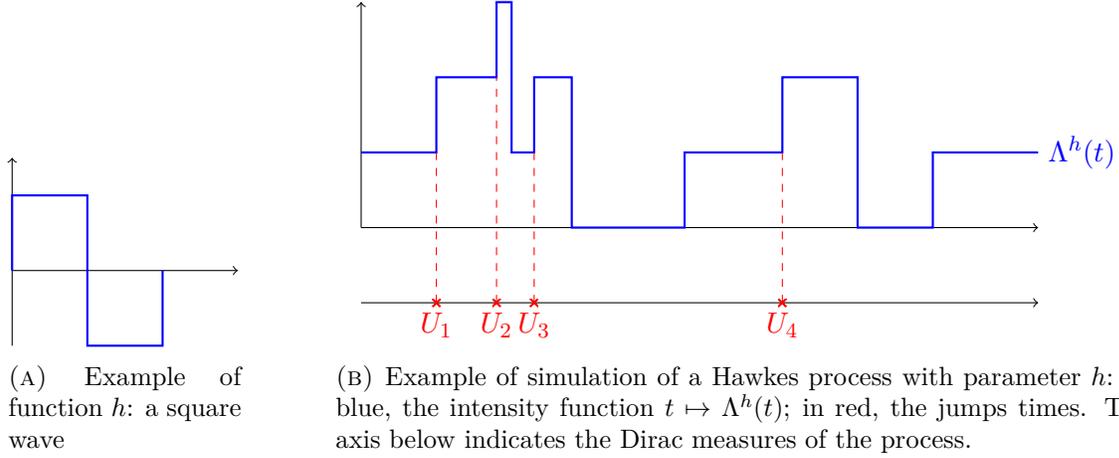
\begin{figure}
		\begin{subfigure}[b]{0.2 \textwidth}
			\centering
			\begin{tikzpicture}
			\draw [->, black] (0, 0) -- ++(3,0);
			\draw [->, black] (0,-1) -- ++(0,2.5);
			
			\draw [thick,blue](0,0) -- ++(0,1) -- ++(1,0) -- ++(0,-2) -- ++(1,0) -- ++(0,1); 
			\end{tikzpicture}
			\caption{Example of function $h$: a square wave}
		\end{subfigure}
		\hspace{1cm}
		\begin{subfigure}[b]{0.7 \textwidth}
			\centering
			\begin{tikzpicture}[scale = 1]
			\draw [->, black] (0, 0) -- (9,0);
			\draw [->, black] (0,0) -- (0,3);
			\draw [->, black] (0, -1) -- (9,-1);
			
			\coordinate (saut1) at (1,0);
			\coordinate (saut2) at (1.8,0);
			\coordinate (saut3) at (2.3,0);
			\coordinate (saut4) at (5.6,0);
			
			\coordinate (mesure1) at ($(saut1) + (0,-1)$);
			\coordinate (mesure2) at ($(saut2) + (0,-1)$);
			\coordinate (mesure3) at ($(saut3) + (0,-1)$);
			\coordinate (mesure4) at ($(saut4) + (0,-1)$);
			
			\coordinate (B1) at ($(saut1) + (0,1)$); 
			\coordinate (B2) at ($(B1) + (0,1)$);
			\coordinate (C1) at ($(saut2) + (0,2)$); 
			\coordinate (C2) at ($(C1) + (0,1)$);
			\coordinate (D1) at ($(saut1) + (1,3)$); 
			\coordinate (D2) at ($(D1) + (0,-2)$);
			\coordinate (E1) at ($(saut3) + (0,1)$); 
			\coordinate (E2) at ($(E1) + (0,1)$);
			\coordinate (F1) at ($(saut2) + (1,2)$); 
			\coordinate (F2) at ($(F1) + (0,-2)$);
			\coordinate (G1) at ($(saut3) + (2,0)$); 
			\coordinate (G2) at ($(G1) + (0,1)$);
			\coordinate (H1) at ($(saut4) + (0,1)$); 
			\coordinate (H2) at ($(H1) + (0,1)$);
			\coordinate (I1) at ($(saut4) + (1,2)$); 
			\coordinate (I2) at ($(I1) + (0,-2)$);
			\coordinate (J1) at ($(saut4) + (2,0)$); 
			\coordinate (J2) at ($(J1) + (0,1)$);
			
			\coordinate (B3) at ($(B1) + (1,1)$); 
			\coordinate (B4) at ($(B1) + (1,0)$);
			\coordinate (B5) at ($(B1) + (2,0)$);
			\coordinate (B6) at ($(B1) + (2,1)$);
			
			\coordinate (C3) at ($(C1) + (1,1)$); 
			\coordinate (C4) at ($(C1) + (1,0)$);
			\coordinate (C5) at ($(C1) + (2,0)$);
			\coordinate (C6) at ($(C1) + (2,1)$);
			
			\coordinate (E3) at ($(E1) + (1,1)$); 
			\coordinate (E4) at ($(E1) + (1,0)$);
			\coordinate (E5) at ($(E1) + (2,0)$);
			\coordinate (E6) at ($(E1) + (2,1)$);

			\draw [thick, blue] (0,1) -- (B1) -- (B2) -- (C1) -- (C2) -- (D1) -- (D2) -- (E1) -- (E2) -- (F1) -- (F2) -- (G1) -- (G2) -- (H1) -- (H2) -- (I1) --(I2) -- (J1) -- (J2) -- (9,1);
			
			\draw [blue] (9,1) node [right] {$\Lambda^h(t)$};
			
			\draw [dashed, red] (mesure1) -- (B1)
			(mesure2) -- (C1)
			(mesure3) -- (E1)
			(mesure4) -- (H1);
			
			\draw [thin, red] plot[mark=x] (mesure1)
			plot[mark=x] (mesure2)
			plot[mark=x] (mesure3)
			plot[mark=x] (mesure4);
			
			\draw [red] (mesure1) node [below] {$U_1$}
			(mesure2) node [below] {$U_2$}
			(mesure3) node [below] {$U_3$}
			(mesure4) node [below] {$U_4$};
			
			\end{tikzpicture}
			\caption{Example of simulation of a Hawkes process with parameter $h$: in blue, the intensity function $t \mapsto \Lambda^h(t)$; in red, the jumps times. The axis below indicates the Dirac measures of the process.}
		\end{subfigure}
		\caption{Example of Hawkes process} \label{fig_ex_processus_Hawkes}
	\end{figure}
\end{center}

Let $N^h$ be a Hawkes process according to Definition \ref{def:Hawkes}, with initial condition $N^0 = \emptyset$. We denote by $U_1, U_2, U_3, ...$ its successive jumps.

Let us introduce the renewal times of the process which splits the time line into independent and identically distributed time windows of length $\tau_1,\tau_2,\cdots$. 

Define the stopping time
$$\tau_1 = \inf \{t> U^1, N^h((t-L(h),t]) = 0\},$$
that is the first time after $U^1$ such that there has been no jump during a time $L(h)$. We also set $$S_0 = 0 \quad \text{and}\quad S_1 = \tau_1.$$
Let us now define  $$W_1 = N^h([U^1, S_1])= N^h([0,S_1]),$$ the number of jumps of the process in this first time window and rename the jump times in the first time window as:
$$U_j^1=U_j,\quad \forall j\in\{1,\cdots, W_1\}.$$ 

We shall see below that $\tau_1$ and $W_1$ are almost surely finite. Recursively let $i \in \N^*$ such that $(\tau_1,W_1), ... (\tau_i,W_i)$ are well defined (and a.s. finite). Let $S_i=\sum_{k=1}^i \tau_k$ and define
$$U^{i+1}_1 = U_{W_1+ ... + W_i + 1},$$
and 
\begin{equation}
\label{def:taui}\tau_{i+1} = \inf \{t> U^{i+1}_1, N^h((t-L(h),t]) = 0\}-S_i,
\end{equation}
Notice that there is at least one jump in $[S_i, S_i + \tau_{i+1}]$.
We finally introduce the number of jumps in the $(i+1)$'th window as 
\begin{equation}\label{def:Wi}
W_{i+1} = N^h([U^{i+1}_1, S_i + \tau_{i+1}) = N^h([S_i, S_i + \tau_{i+1}]),\end{equation}
and rename the associated jump times as:
$$U^{i+1}_j = U_{W_1+ ... + W_i + j}, \quad \forall j \in \{1, ..., W_{i+1}\}. $$

Figure \ref{fig_ex_creneau_neg_retard} is an example of this splitting of the time and the renumbering of the jumps, in the case where $h(t) = - \lambda \1_{(1, 2)}(t)$, so that $L(h) = 2$.

\begin{figure}
	\centering
	\begin{tikzpicture}[scale=1.5]
	\draw [->, black] (0, 0) -- (9,0);
	\draw [->, black] (0,0) -- (0,1.5);
	\draw [->, black] (0, -1) -- (9,-1);
	
	\coordinate (saut1) at (0.5,0); 
	\coordinate (saut2) at (0.7,0);
	\coordinate (saut3) at (0.8,0);
	\coordinate (saut4) at (1.2,0);
	
	\coordinate (saut5) at (3.5,0); 
	
	\coordinate (saut6) at (5.9,0); 
	\coordinate (saut7) at (6.2,0);
	
	\coordinate (s1_debut_haut) at ($(saut1) + (1,1)$);
	\coordinate (s5_debut_haut) at ($(saut5) + (1,1)$);
	\coordinate (s6_debut_haut) at ($(saut6) + (1,1)$);
	
	\coordinate (s1_debut_bas) at ($(saut1) + (1,0)$);
	\coordinate (s5_debut_bas) at ($(saut5) + (1,0)$);
	\coordinate (s6_debut_bas) at ($(saut6) + (1,0)$);
	
	\coordinate (s1_fin_bas) at ($(saut4) + (2,0)$);
	\coordinate (s5_fin_bas) at ($(saut5) + (2,0)$);
	\coordinate (s6_fin_bas) at ($(saut7) + (2,0)$);
	
	\coordinate (s1_fin_haut) at ($(s1_fin_bas) + (0,1)$);
	\coordinate (s5_fin_haut) at ($(s5_fin_bas) + (0,1)$);
	\coordinate (s6_fin_haut) at ($(s6_fin_bas) + (0,1)$);

	\coordinate (mesure1) at ($(saut1) + (0,-1)$);
	\coordinate (mesure2) at ($(saut2) + (0,-1)$);
	\coordinate (mesure3) at ($(saut3) + (0,-1)$);
	\coordinate (mesure4) at ($(saut4) + (0,-1)$);
	\coordinate (mesure5) at ($(saut5) + (0,-1)$);
	\coordinate (mesure6) at ($(saut6) + (0,-1)$);
	\coordinate (mesure7) at ($(saut7) + (0,-1)$);
	
	\draw [thick, blue] (0,1) -- (s1_debut_haut) -- (s1_debut_bas) -- (s1_fin_bas) -- (s1_fin_haut) -- (s5_debut_haut) -- (s5_debut_bas) -- (s5_fin_bas) -- (s5_fin_haut) -- (s6_debut_haut) -- (s6_debut_bas) -- (s6_fin_bas) -- (s6_fin_haut) -- (9,1);
	
	\draw [blue] (0,1) node [above right] {$\Lambda^h(t)$};
	
	\draw [thin, red] plot[mark=x] (mesure1)
	plot[mark=x] (mesure2)
	plot[mark=x] (mesure3)
	plot[mark=x] (mesure4)
	plot[mark=x] (mesure5)
	plot[mark=x] (mesure6)
	plot[mark=x] (mesure7);
	
	\draw [dashed, red] (mesure1) -- ++(0,2)
	(mesure2) -- ++(0,2)
	(mesure3) -- ++(0,2)
	(mesure4) -- ++(0,2)
	(mesure5) -- ++(0,2)
	(mesure6) -- ++(0,2)
	(mesure7) -- ++(0,2);
	
	\draw [red] (mesure1) node [below] {$U^1_1$}
	(mesure3) node [below = 10 pt] {$...$}
	(mesure4) node [below] {$U^1_4$}
	(mesure5) node [below] {$U^2_1$}
	(mesure6) node [below] {$U^3_1$}
	(mesure7) node [below right] {$U^3_2$};
	
	\draw [red, decoration={brace,amplitude=5 pt}, decorate] ($(mesure4)+(0.2,-0.4)$) -- ($(mesure1)+(-0.2,-0.4)$) coordinate[midway] (mid);
	\draw [red] (0.8, -1.7) node {$W_1 =4$};
	
	\draw [red, decoration={brace,amplitude=5 pt}, decorate] ($(mesure5)+(0.2,-0.4)$) -- ($(mesure5)+(-0.2,-0.4)$) coordinate[midway] (mid);
	\draw [red] ($(mesure5)+(0, -0.7)$) node {$W_2 =1$};
	
	\draw [red, decoration={brace,amplitude=5 pt}, decorate] ($(mesure7)+(0.4,-0.4)$) -- ($(mesure6)+(-0.2,-0.4)$) coordinate[midway] (mid);
	\draw [red] (6.1, -1.7) node {$W_3 =2$};
	
	\draw [<->, darkgray]  (1.6, 1) -- (3.1, 1);
	\draw [darkgray] (2.3, 1) node [above] {$1+U^1_4- U^1_1$};
	
	\draw [<->, darkgray]  (4.6, 1) -- (5.4, 1);
	\draw [darkgray] (5, 1) node [above] {$1$};
	
	\draw [<->, darkgray]  (7, 1) -- (8.1, 1);
	\draw [darkgray] (7.5, 1) node [above] {$1+U^3_2-U^1_1$};
	
	\draw plot[mark=+] (0, -1);
	\draw (0, -1) node [below] {$0$};
	
	\draw plot[mark=+] (3.2, -1);
	\draw ($(s1_fin_bas) + (0, -1)$) node [below] {$S_1$};
	
	\draw plot[mark=+] (5.5, -1);
	\draw ($(s5_fin_bas) + (0, -1)$) node [below] {$S_2$};
	
	\draw plot[mark=+] (8.2, -1);
	\draw ($(s6_fin_bas) + (0, -1)$) node [below] {$S_3$};
	
	\draw [dashed, black] (0, -1) -- (0,0)
	 ($(s1_fin_bas)$) -- ++ (0, -1)
	 ($(s5_fin_bas)$) -- ++ (0, -1)
	 ($(s6_fin_bas)$) -- ++ (0, -1) ;
	 
	\draw [<->, violet, semithick] (0, -0.5) -- ($(s1_fin_bas) + (-0.01, -0.5)$);
	\draw [<->, violet, semithick] ($(s1_fin_bas) + (0.01, -0.5)$) -- ($(s5_fin_bas) + (-0.01, -0.5)$);
	\draw [<->, violet, semithick] ($(s5_fin_bas) + (0.01, -0.5)$) -- ($(s6_fin_bas) + (-0.01, -0.5)$);
	
	\draw [violet] (2.2, -0.5) node [below] {$\tau_1$};
	\draw [violet] (4.5, -0.5) node [below] {$\tau_2$};
	\draw [violet] (6.7, -0.5) node [below] {$\tau_3$};
	\end{tikzpicture}
	\caption{Example of the evolution of intensity in function of time and renumbering of jumps in the case where $h=-\lambda \, \1_{[1,2]}$.} \label{fig_ex_creneau_neg_retard}
\end{figure}
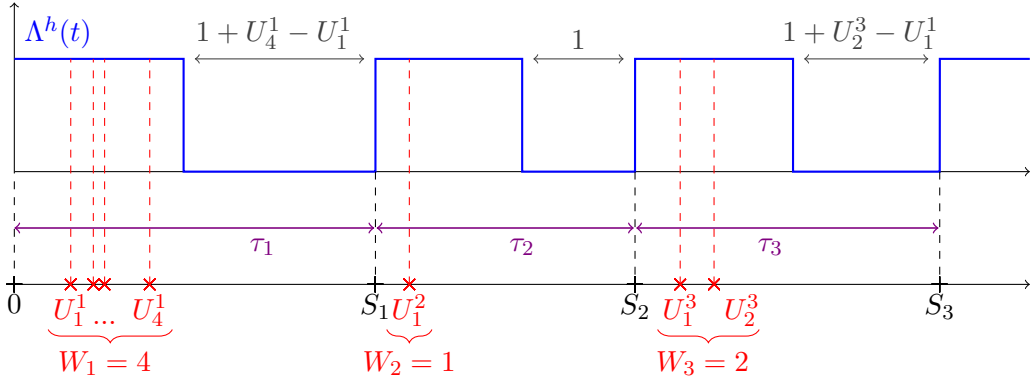

The next Proposition gathers important properties on the law of $(\tau_i,W_i)$ defined above. However more explicit information are difficult to obtain except in specific cases (see Section \ref{secexample}).

\begin{prop}\label{lem_tauW_iid_loiexp} Under Assumptions \ref{assumptions}, and using the above definitions:
\begin{itemize}
\item[$i)$] the $(\tau_i, W_i)_i$ are i.i.d. random variables,
\item[$ii)$] for $i \in \N^*$, the $(U^i_1 - S_{i-1})$ are i.i.d. random variables with exponential distribution $\Exp(\lambda)$, that is, the time between the beginning of a window and the first point of this window follows an exponential law.
\end{itemize}
\end{prop}
\begin{proof}
Let $Q$ be a two-dimensional Poisson point process, and let $N^h$ generated by $Q$ as in proposition \ref{prop_EDS}, the $(\tau_i, W_i)_i$ being defined as before.

Given $\tau_1$, remark that $U^2_1$ is the first jump of $Q$ on $(\tau_1, +\infty) \times [0, \infty]$. Indeed, using successively the definition of $L(h)$ and $\tau_1$ we deduce that:
\begin{align*}
\Lambda(\tau_1) &= \left( \lambda +\int_{(-\infty,\tau_1)} h(t-u)N^h(du)\right)^+\\
&=\left( \lambda +\int_{(\tau_1-L(h),\tau_1)} h(t-u)N^h(du)\right)^+\\
&= \lambda.
\end{align*}

By translation, $U^2_1 - \tau_1$ is the first jump of a Poisson point process $Q'$ on $(0, \infty) \times [0, \infty]$, independent of $Q$ on $(0, \tau_1) \times [0, \infty]$, and $U^2_1 - \tau_1$ is independent of $\tau_1 = S_1$. \\
Since the jumps of $N^h$ before $\tau_1$ do not influence  $\Lambda^h(t)$ for $t > \tau_1$ (by definition of $L(h)$ and $\tau_1$), $$\tau_{2} = \inf \{t> U^{2}_1 - S_1, N^h((t+S_1-L(h),t+S_1]) = 0\},$$ only depends on $Q$ on $(U^2_1, +\infty) \times [0, \infty]$. Moreover, $(0, S_1) \times [0, \infty]$ and $(U^2_1, +\infty) \times [0, \infty]$ are almost surely disjoints. Hence $Q$ on $(U^2_1, +\infty) \times [0, \infty]$ is independent of $Q$ on $(0, S_1) \times [0, \infty]$ so that  $\tau_2$ is independent of $(\tau_1,W_1)$. \\
The number of points in the second time window $W_{2} = N^h([U^{2}_1, S_{2}]) = N^h([S_1, S_{2}])$ only depends on $Q$ on $(U^2_1, +\infty) \times [0, \infty]$. $W_1$ depends on $Q$ on $(0, S_1) \times [0, \infty]$. For the same reason as before, $W_2$ is independent of $(\tau_1,W_1)$. The same argument can be used for each $k$: as the $(S_k)_k$ split $\R^+$ in disjoints intervals, then $Q$ on each of these intervals is independent of $Q$ of another interval. 
	
	In particular, $U^1_1 = U^1_1 -S_0$, $U^2_1 - S_1$ (and all the following) are independent and can be defined as the first jump of a Poisson point process on $(0, +\infty) \times [0, \infty]$. Then they follows an exponential law of parameter $\lambda$. \\
	Using time translation, we see that $\tau_1$, $\tau_2$ (and so on) are defined the same way and follow the same law. Then $W_1$, $W_2$ (and so on) are defined the same way and follow the same law. 
	\end{proof}
This construction indicates the renewal structure generated by the Hawkes process. We shall use this structure to prove limit theorems.
\medskip

To this end remark that 
\begin{equation}\label{eqcumul1}
N_t^h = \sum_{i=1}^{\infty} \, \1_{U_i \leq t} = \sum_{i=1}^{\infty} \, \sum_{j=1}^{W_i} \1_{U_i^j \leq t} \, .
\end{equation}
Introduce the renewal process associated to the $S_i$'s
\begin{equation}\label{eqcumul2}
M_t^h := \, \sum_{i=1}^{\infty} \, \1_{S_i\leq t} \, .
\end{equation}
Since $S_i=\sum_{k=1}^i \, \tau_k$ we may introduce 
\begin{equation}\label{def:hatN}\hat{N}^h_t := \sum_{i =1}^{\infty} W_i \1_{S_{i} \leq t} = \sum_{i=1}^{M^h_t} W_i. \end{equation}
For any $t \in \R^+$, the current window is the $M^h_t+1$-th. $\hat{N}^h_t$ includes only the jumps up to the $M^h_t$-th window, while $N^h_t$ can have more jumps. In particular, 
\begin{equation}\label{eqcumul3}
\hat{N}^h_t \leq N^h_t \leq \hat{N}^h_t + W_{M^h_t+1} \quad \text{a.s.}
\end{equation}
We thus have $$N_t^h = \sum_{i=1}^{M^h_t} W_i + R_t^h$$ for some renewal process $M_t^h$ and a remaining term  $R_t^h \leq W_{M^h_t+1}$, the $W_i$'s being i.i.d.. Such processes are known as \textit{cumulative processes} in the literature. A LLN and a CLT for $\hat{N}_t^h$ can be found in \cite{asmussen} theorems 3.1 and 3.2. The LD principle for cumulative processes is studied in \cite{lefevld} in the special case $$W_i=F(\tau_i)$$ for some non-negative, bounded and continuous function $F$ (see the references in \cite{lefevld} for some previous results in still more specific cases). These results do not apply for Hawkes processes, and we had to establish a more general LD principle in the companion paper \cite{CCC2}. In order to get similar results for $N^h_t/t$ it will remain to study the remaining $R_t^h$.
\medskip

\subsection{LLN, CLT and asymptotic deviations for signed reproduction function $h$. \\ \\} 
\label{subsec_results_LFGN_neg_h}

\noindent We now state the main results of the paper. The key is to get enough moments for the $(\tau_i,W_i)$. The first result deals with this problem.

\begin{prop}
\label{prop:tau_w_h_negative}
Let $h$ be a signed function satisfying Assumptions \ref{assumptions}. Let us consider the Hawkes process $N^h$ and the i.i.d.  couples of random variables $(\tau_i,W_i)$ defined in \eqref{def:taui}-\eqref{def:Wi}. 
\begin{itemize}
\item[$i)$] For $\alpha < \, \alpha_0:= \min \left(\lambda \; , \; \frac{||h^+||_1 - \ln(||h^+||_1) - 1}{L(h)}\right)$ we have $\E(e^{\alpha \tau_1}) < +\infty$. 
\item[$ii)$] There exists $\theta_0 >0$ such that for $\theta<\theta_0$, $\mathbb E(e^{\theta W_1}) < +\infty$.
\end{itemize} 
In particular $\tau_1$ and $W_1$ have polynomial moments of any order. 
\end{prop}
The proof of this proposition is given in Section \ref{sec_proof_moment}. Actually, one can give an lower bound on $\theta_0$. This lower bound differs whether $h\le 0$ or not. \\
In the general case, the upper bound for $\theta_0$ depends on a random variable $\mathcal{S}$ with distribution $$\mathbb P(\mathcal{S}=k) = \frac{e^{-k\|h^+\|_1}(k\|h^+\|_1)^{k-1}}{k!}. $$
Using a comparison with a queue process that will be detailed in the proof of Proposition \ref{prop:tau_w_h_negative} below one prove that  $\theta_0$ can be chosen as
$$\theta_0 < ||h^+||_1 - \ln(||h^+||_1) - 1\quad \text{and}\quad \lambda (\mathbb E(e^{2 \theta_0 \mathcal{S}})-1)< \alpha_0.$$
In the case of pure inhibition, i.e. $h\leq 0$, the quantity $||h^+||_1 - \ln(||h^+||_1)$ becomes infinite. However using a comparison with a Poisson Process one can get another explicit bound for $\theta_0$, whose proof will also be given in Section \ref{sec_proof_moment}.
\begin{prop}\label{propWexpnegative}
If $h \leq 0$, one can choose $\theta_0 < - \, \ln\left(1- e^{- \lambda L(h)}\right)$ in proposition \ref{prop:tau_w_h_negative}.
\end{prop}
\begin{rmq} \label{rem:cassimple}
Exacts computations for moments of $\tau$ and $W$ are difficult. Let us consider here and in section \ref{secexample} some specific cases.
\smallskip

Notice that for $h=0$ (i.e. in the case of a Poisson process), $W_1=1$ has exponential moments of any order and $\tau_1$ whose distribution is exponential with parameter $\lambda$, has exponential moments up to order $\lambda$.
\medskip

Another basic case is the \emph{canceling of intensity} case, i.e. choosing the reproduction function as  $g=-\lambda \, \1_{[0,A]}$ for some positive $\lambda$ and $A$. We have seen in Proposition \ref{prop_minoration} that the corresponding $N_t^g$ is smaller than any $N_t^h$ with $L(h)=A$. Since for $t \in (U_1^1,U_1^1+A)$ it holds $\Lambda^h(t)=0$, it immediately follows that $\tau_1=U_1^1+A$ and $W_1=1$, so that $$(W_1,\tau_1)\sim (1, A+\mathcal E(\lambda)) \, ,$$ so that $\mathbb E(\tau_1)=A + \lambda^{-1}$, $\Var(\tau_1)=\lambda^{-2}$, $\alpha_0=\lambda$ and $\theta_0=+\infty$. \hfill $\diamondsuit$
\end{rmq} 
\medskip

From these moments properties and the renewal structure of the Hawkes process, we will derive the following asymptotic results:
\begin{thm}[\bf{Law of Large Numbers}]
\label{thm:lln}
Let $h$ be a signed function satisfying Assumptions \ref{assumptions} and  consider the Hawkes process $N^h$ given by \eqref{eq_EDS}.
Then we have the following:  
	$$ \frac{N^h_t}{t} \cvg[a.s.]{t}{\infty} \frac{\E[W_1]}{\mathbb E(\tau_1)} \, . $$ 
\end{thm}	
Thanks to our comparison results and to \eqref{eq:lln_positive} we have $$\frac{\lambda}{1+\lambda L(h)} \, \leq \, \frac{\E[W_1]}{\mathbb E(\tau_1)} \, \leq \, \frac{\lambda}{1 - ||h^+||_1} \, .$$
\medskip

Our method will also provide us with a CLT.
\begin{thm}[\bf{Central Limit Theorem}]
\label{thm:clt}
Let $h$ be a signed function satisfying Assumptions \ref{assumptions} and  consider the Hawkes process $N^h$ given by \eqref{eq_EDS}.
Then 
	$$ \sqrt t \, \left(\frac{N^h_t}{t} \, -  \, \frac{\E[W_1]}{\mathbb E(\tau_1)}\right) \underset{t\to\infty}{\Longrightarrow} \mathcal{N}(0,\sigma^2)$$with $$\sigma^2 = \frac{\Var\left(W_1-\tau_1 \, \frac{\E[W_1]}{\mathbb E(\tau_1)}\right) }{\mathbb E(\tau_1)} \, .$$ 
\end{thm}
\medskip

We finally state deviation results based on the results in the companion paper \cite{CCC2}, in which we obtain large deviation for general cumulative processes. To this end we need to introduce some notations
\begin{definition}\label{defcramer}
We introduce the Cramer transform for $(a,b) \in \mathbb R^2$, $$\Lambda^*(a,b)= \sup_{x,y} \left\{ax + by - \ln \left(\mathbb E\left[e^{x\tau_1+y W_1}\right]\right)\right\} \, .$$ 

We also define for $z \in \mathbb R^+$, $$J(z) = \inf_{\beta>0} \, \left(\beta \, \Lambda^*\left(\frac{1}{\beta} \, , \, \frac{z}{\beta}\right)  \right) \, .$$ Similarly we define $\Lambda^*_n$ and $J_n$ replacing $W_1$ by $\min(W_1,n)$. Finally define $$\tilde J(z) = \sup_{\delta>0} \, \liminf_{n \to \infty} \, \inf_{|y-z|<\delta} J_n(y) \, .$$
\end{definition}
Thanks to Proposition \ref{prop:tau_w_h_negative} we may apply Theorem 2.4 in \cite{CCC2}, telling us that the distributions of $\hat N^h_t/t$ satisfy asymptotic deviation inequalities 
\begin{thm}\label{thmLD}
Recall that $\theta_0$ is defined in Proposition \ref{prop:tau_w_h_negative} (ii).
\begin{itemize}
\item If $\theta_0=+\infty$, the laws of the family $\hat N^h_t/t$ satisfy a large deviation principle with rate function $\tilde J$, i.e
\begin{enumerate} 
\item for any closed set $\mathcal C \in \mathbb R$, $$\limsup_{t \to \infty} \, \frac 1t \; \ln \mathbb P\left(\hat N^h_t/t \, \in \, \mathcal C\right) \, \leq \, - \, \inf_{m\in \mathcal C} \, \tilde J(m)$$
\item for any open set $\mathcal O \in \mathbb R$, $$\liminf_{t \to \infty} \, \frac 1t \; \ln \mathbb P\left(\hat N^h_t/t \, \in \, \mathcal O\right) \, \leq \, - \, \inf_{m\in \mathcal O} \, \tilde J(m) \, ,$$
\end{enumerate}
\item If $\theta_0<+\infty$, denoting $m=\mathbb E(W_1)/\mathbb E(\tau_1)$ we have for all $a>0$ $$\limsup_{t \to +\infty} \, \frac 1t \, \ln \mathbb P\left(\frac{\hat N^h_t}{t} > m +a\right) \leq - \, \min\left[\inf_{z\geq m+(a/2)} J(z) \; , \; \theta_0 a/2 \right] \, ,$$ and $$\limsup_{t \to +\infty} \, \frac 1t \, \ln \mathbb P\left(\frac{\hat N^h_t}{t} < m -a\right) \leq - \, \min\left[\inf_{z\geq m-(a/2)} J(z) \; , \; \theta_0 a/2 \right] \, .$$ We may replace $a/2$ and $\theta_0 a/2$ by $\kappa a$ and $(1-\kappa) \theta_0 a$ for any $\kappa \in (0,1)$.
\end{itemize}
\end{thm}
The latter deviation inequalities are obtained using that $J \leq \tilde J$ (see \cite{CCC2}).
\smallskip

We finally will prove
\begin{cor}\label{corLD}
Recall that $\theta_0$ is defined in Proposition \ref{prop:tau_w_h_negative} (ii).
\begin{enumerate}
\item[(1)] \quad If $\theta_0=+\infty$, $N^h_t/t$ satisfies the same LDP as $\hat N^h_t/t$.
\item[(2)] \quad If $\theta_0<+\infty$, we have for all $a> 0$
\begin{equation}\label{eqaudessus}
\limsup_{t \to \infty} \, \frac 1t \; \ln \mathbb P\left(\frac{N^h_t}{t} \,  > \, m + a \right) \leq \, - \, \min \left[\inf_{z-m \geq   \kappa a} \, J(z) \, , \, \kappa' \theta_0 a\right] \, ,
\end{equation}
with $\kappa$ and $\kappa'$ in $(0,1)$ satisfying $\kappa +2 \kappa'=1$. 
Similarly
\begin{equation}\label{eqendessous}
\limsup_{t \to \infty} \, \frac 1t \; \ln \mathbb P\left(\frac{N^h_t}{t} \,  < \, m - a \right) \leq \, - \, \min \left[\inf_{m-z \leq   \kappa a} \, J(z) \, , \, (1-\kappa)  \theta_0 a \right] \, .
\end{equation}
for $\kappa \in (0,1)$.
\end{enumerate}
\end{cor}
\medskip

\begin{rmq}\label{LDcancel}
Once again we may get an explicit expression for the rate function in the canceling intensity case $h= - \lambda \, \1_{[0,A]}$. Since $W_1=1$ and $\tau_1-A$ is an exponential variable with parameter $\lambda$, we have 
$$
\beta \Lambda^*\left(\frac 1\beta,\frac m\beta\right)=\sup_{x,y} \left(x+(m-\beta)y + \beta \ln \left(1- \frac{x}{\lambda}\right) - \beta Ax \right) \, .
$$
Notice that for a given $x$, $\sup_y \left(x+(m- \beta)y + \beta \ln \left(1- \frac{x}{\lambda}\right) - \beta Ax \right) < +\infty$ if and only if $\beta=m$ due to the linear term in $y$. We deduce $$J(m) = m \, \Lambda^*\left(\frac 1m,1\right) \, .$$
It  easily follows $$J(m)=\lambda(1-mA) - m + m \, \ln\left(\frac{m}{\lambda(1-mA)}\right) \, .$$
\hfill $\diamondsuit$
\end{rmq}

\bigskip

\section{One more example with explicit calculations: canceling intensity with delay.}
\label{secexample}

We already discussed in Remark \ref{rem:cassimple} the \emph{canceling of intensity} case $h=-\lambda \, \1_{[0,A]}$. 

In our second example we add a delay to the previous case: the inhibition only occurs after a lag period of length $r>0$. 
Let $\lambda >0$, $r>0$ and $A>r$ we consider $h = -\lambda \1_{[r, r+A]}$. Then $L(h) = r +A$. We can again explicitly compute the law of $W_i$ and $\tau_i$.

\begin{center}
	\begin{figure}[h!]
		\begin{subfigure}[b]{0.45 \textwidth}
			\centering
			\begin{tikzpicture}
				\draw [->, black] (0, 0) -- (6.4,0);
				\draw [->, black] (0,0) -- (0,2);
				\draw [->, black] (0, -0.5) -- (6.4,-0.5);
				
				\draw [black] (0, 2) node [right] {$\Lambda^h(t)$};
				\draw [black] (6.4, 0) node [right] {$t$};
				\draw [black] (6.4, -0.5) node [right] {$N^h$};
				\draw [black] (0, 1) node [left] {$\lambda$};
				\draw plot[mark=+] (0, 1) ;
				
				\coordinate (saut1) at (0.5,0);
				\coordinate (saut2) at (1.97,0);
				\coordinate (saut3) at (4.02,0);
				\coordinate (saut4) at (5.6,0);

				\draw [thick, blue] (0,1) -- ($(saut1)+(0,1)$) -- ($(saut1)$) -- ($(saut1) + (1.3,0)$) -- ($(saut1) + (1.3,1)$) -- ($(saut2)+(0,1)$) -- ($(saut2)$) -- ($(saut2) + (1.3,0)$) -- ($(saut2) + (1.3,1)$) -- ($(saut3)+(0,1)$) -- ($(saut3)$) -- ($(saut3) + (1.3,0)$) -- ($(saut3) + (1.3,1)$) -- ($(saut4)+(0,1)$) -- ($(saut4)$) -- (6.4, 0);
				\draw [<->,dashed,darkgray] ($(saut1)+(0,0.5)$) --  ($(saut1)+(1.3,0.5)$) ;
				\draw [black] (1.1, 0.8) node  {$A$};
				
				\draw [dashed, red] ($(saut1) + (0, -0.5)$) -- ($(saut1) + (0, 1)$)
				($(saut2) + (0, -0.5)$) -- ($(saut2) + (0, 1)$)
				($(saut3) + (0, -0.5)$) -- ($(saut3) + (0, 1)$)
				($(saut4) + (0, -0.5)$) -- ($(saut4) + (0, 1)$);
				
				\draw [thin, red] plot[mark=x] (0.5, -0.5)
				plot[mark=x] (1.97, -0.5)
				plot[mark=x] (4.02, -0.5)
				plot[mark=x] (5.6, -0.5);
			\end{tikzpicture}
			\caption{Example of Hawkes process : canceling intensity without delay, $h=-\lambda \, \1_{[0,A]}$.}
		\end{subfigure}
		\hspace{1cm}
		\begin{subfigure}[b]{0.45\textwidth}
			\centering
			\begin{tikzpicture}[scale = 1]
			\draw [->, black] (0, 0) -- (6.4,0);
			\draw [->, black] (0,0) -- (0,2);
			\draw [->, black] (0, -0.5) -- (6.4,-0.5);
			
			\draw [black] (0, 2) node [right] {$\Lambda^h(t)$};
			\draw [black] (6.4, 0) node [right] {$t$};
			\draw [black] (6.4, -0.5) node [right] {$N^h$};
			\draw [black] (0, 1) node [left] {$\lambda$};
			\draw plot[mark=+] (0, 1) ;
			
			\coordinate (saut1) at (0.2,0);
			\coordinate (saut2) at (0.41,0);
			\coordinate (saut3) at (0.74,0);
			\coordinate (saut4) at (0.92,0);
			
			\coordinate (saut5) at (3.1,0);
			
			\coordinate (saut6) at (5.47, 0);
			\coordinate (saut7) at (5.9,0);			
			
			\draw [thick, blue] (0,1) -- ($(saut1) + (0.8,1)$) -- ($(saut1) + (0.8,0)$) -- ($(saut4) + (2.1,0)$) -- ($(saut4) + (2.1,1)$) -- ($(saut5) + (0.8,1)$) -- ($(saut5) + (0.8,0)$) -- ($(saut5) + (2.1,0)$) -- ($(saut5) + (2.1,1)$) -- ($(saut6) + (0.8,1)$) -- ($(saut6) + (0.8,0)$) -- (6.4, 0);
			
			\draw [<->,dashed,darkgray] ($(saut1)+(0,1.1)$) --  ($(saut1)+(0.8,1.1)$) ;
				\draw [black] (0.7, 1.25) node  {$r$};
			\draw [dashed, red] ($(saut1) + (0, -0.5)$) -- ($(saut1) + (0, 1)$)
			($(saut2) + (0, -0.5)$) -- ($(saut2) + (0, 1)$)
			($(saut3) + (0, -0.5)$) -- ($(saut3) + (0, 1)$)
			($(saut4) + (0, -0.5)$) -- ($(saut4) + (0, 1)$)
			($(saut5) + (0, -0.5)$) -- ($(saut5) + (0, 1)$)
			($(saut6) + (0, -0.5)$) -- ($(saut6) + (0, 1)$)
			($(saut7) + (0, -0.5)$) -- ($(saut7) + (0, 1)$);
			
			\draw [thin, red] plot[mark=x] (0.2, -0.5)
			plot[mark=x] (0.41, -0.5)
			plot[mark=x] (0.74, -0.5)
			plot[mark=x] (0.92, -0.5)
			plot[mark=x] (3.1, -0.5)
			plot[mark=x] (5.47,-0.5)
			plot[mark=x] (5.9,-0.5);
			\end{tikzpicture}
			\caption{Example of Hawkes process : canceling intensity with a delay, $h=-\lambda \, \1_{[r,r+A]}$.}
		\end{subfigure}
		\caption{Comparison of Hawkes processes with or without a delay in the canceling of the intensity:\\ In blue, the intensity function $t \mapsto \Lambda^h(t)$; in red, the jumps times. The axis below indicates the Dirac measures of the process.}
		\label{fig_comparison_cancel}
	\end{figure}
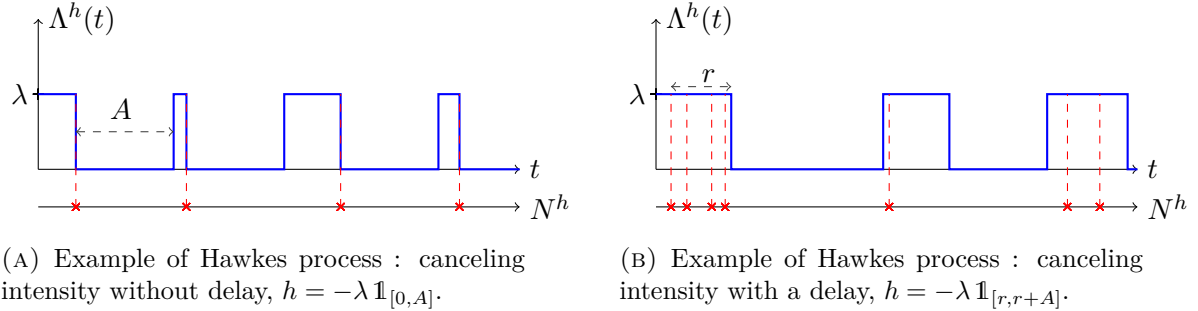
\end{center}

We can summarize the results of this two cases and apply Theorem \ref{thm:lln} to obtain 
\begin{prop}\label{prop:examples}
Let us consider $A>0$ and $r\ge 0$.
The Hawkes process associated with $h=-\lambda\, \1_{[r,r+A]}$ satisfies
$$\lim_{t\to\infty}\frac{N^{h}_t}{t}= \frac{\lambda(1+\lambda r)}{\lambda A +2\lambda t +\e^{-r}}\quad a.s.$$
\end{prop}
\begin{rmq} \quad This result naturally leads to some comments on the issues brought by inhibition.
\begin{itemize}
\item Let us first remark that as $r\to 0$ we recover the result of the canceling intensity case given in Remark \ref{rem:cassimple}.
\item Secondly we wonder whether one of both examples admits more points asymptotically. Therefore we are lead to study the ratio
$$\frac{\frac{\lambda}{\lambda A+1}}{\frac{\lambda(1+\lambda r)}{\lambda A +2\lambda r +e^{-\lambda r}}}=\frac{\lambda A +2\lambda r +e^{-\lambda r}}{(\lambda A+1)(1+\lambda r)},$$
or equivalently the sign of 
$$e^{-\lambda r} - 1+\lambda r -\lambda^2Ar=\lambda^2 r \left(\frac{r}{2}-A\right)+\sum_{k=3}^\infty \frac{(-\lambda r)^k}{k!}.$$ 
using the series expansion of the exponential. We therefore deduce that since $A>r$, the right hand side is negative, and thus the ratio is less that $1$. Consequently, this proves that the lag induces asymptotically more points in the Hawkes process. Notice that $||h||_1$ is the same in both cases. \\
From this we deduce that the formula \eqref{eq:lln_positive} is not true in the case of signed  reproduction function $h$. 
\end{itemize}
\hfill $\diamondsuit$
\end{rmq}

\begin{proof}[Proof of Proposition \ref{prop:examples}]
 Let us study $\Lambda^h$ on the time interval $[U_1^1, U_1^1+r+A]$:
\begin{itemize}
\item for $t \in [U^1_1, U^1_1+r)$, then for any $u\in(0,t)$, $t-u$ belongs to $(0,t)$ and thus 
$$\Lambda^h(t) = \left( \lambda + \int_0^t - \lambda \1_{(t-u) \in [r, r+A]} N^h(du) \right)^+ = \lambda,$$
\item for $t \in [U^1_1 +r, U^1_1 + r +A]$, then $$\Lambda^h(t)= \left( \lambda + \int_{(0, t)} - \lambda \1_{(t-u) \in [r, r+A} N^h(du) \right)^+ \leq \left( \lambda - \lambda \1_{(t-U^1_1) \in [r, r+A]}  \right)^+ = 0 \, .$$
\end{itemize}
From this, we deduce that all the points of $N^h$ in $]U_1^1, U_1^1+r+A]$ actually belong to the interval $]U_1^1, U_1^1+r[$. In particular, if $N^h$ has no points in $]U_1^1, U_1^1+r+A]$, then $W_1=1$ and $\tau_1=U_1^1+r+A$.\\
Let us now remark that $N^h(]U_1^1, U_1^1+r])$ follows a Poisson law of parameter $\lambda r$ since the intensity of Hawkes process is constant on this interval. In particular $N^h([U_1^1, U_1^1+r])$ is finite almost surely. 
More generally for any $1 < k \leq N^h([U_1^1, U_1^1+r])$, then $U^1_{k} \leq U^1_1 + r$ and 
$$\forall t \in [U^1_k +r, U^1_k + r +A],\quad \Lambda^h(t) = 0.$$ 
Finally, since $A>r$ we have that $U^1_k+r \le U_1^1+r+r\le U_1^1+r+A$, and thus the intensity remains null on the interval $[U_1^1+r, U^1_k + r +A]$.
\\ We can conclude that 
\begin{align}
&W_1=N^h([U_1^1, U_1^1+r]),\label{eq:ex2_W}\\
&\tau_1=U_{W_1}^1 + r+A.\label{eq:ex2_tau}
\end{align}
%
%
%
Since the first point in $N^h $ in the interval $[U^1_1, U^1_1+r]$ is in $U_1^1$ we actually have
$$W_1 = 1 + N^h((U^1_1, U^1_1+r)).$$
It follows that $$W_1 - 1 \sim \mathcal{P}(\lambda r)$$ and  $$\E(W_1)=1 + \lambda r\,, \quad \Var(W_1)= \lambda r \quad \textrm{ and } \quad \theta_0=+\infty \, .$$

We finally study the law of $\tau_1$. From Equation \eqref{eq:ex2_tau}, we can write 
\begin{align*}
\tau_1 &= U^1_{W_1} + r+A\\
&= r+A+ U^1_1 + (U^1_{W_1} - U^1_1),
\end{align*} 
where $U^1_1 \sim \Exp(\lambda)$ by lemma \ref{lem_tauW_iid_loiexp} and $U^1_1$ and $(U^1_{W_1} - U^1_1)$ are independent.\\ It remains to study the law of $(U^1_{W_1} - U^1_1)$. 

Thanks to \eqref{eq:ex2_W}, $0\le U^1_{W_1}-U^1_1\le r$.

Let $t \in [0, r]$, we have: 
$$\Proba \left( 0 \leq U^1_{W_1} - U^1_1 \leq t \right)= \sum_{k=1}^{\infty} \Proba \left( \{W_1 = k \} \cap \{U^1_{W_1} - U^1_1 \leq t \} \right).$$ For $k=1$: $\Proba \left( \{W_1 = 1\} \cap \{U^1_{W_1} - U^1_1 \leq t\} \right) = \Proba \left( W_1 = 1 \right) = \e^{-\lambda r}$.

For $k >1$ since the intensity of the Hawkes process remains constant equal to $\lambda$ on $[U^1_1, U^1_{W_1}]$ we can write 
$$U^1_{W_1} - U^1_1 \overset{(law)}{=}\sum_{k=1}^{W_1-1} T_k$$ where $(T_k)_{k\in\N}$ is a sequence of i.i.d  $\mathcal{E}(\lambda)$. We can consider $(T_k)_{k\in\N}$ as the interarrival times of a Poisson process of parameter $\lambda$ coupled with our Hawkes process, as in Proposition \ref{prop_EDS}. Then, $T_0 = U^1_1$, and $(T_k)_{k\geq W_1}$ are defined. Then: 
\begin{align*}
\Proba \left( \{W_1 = k \} \right.&\left.\cap \{U^1_{W_1} - U^1_1 \leq t\} \right)
= \Proba\left( \left\lbrace 0 \leq \sum_{i=1}^{k-1} T_i \leq t \right\rbrace \cap \left\lbrace T_{k} + \sum_{i=1}^{k-1} T_i > r \right\rbrace \right) \\
&= \E \left[ \1_{0 \leq \sum_{i=1}^{k-1} T_i \leq t} \Proba\left( T_{k} + \sum_{i=1}^{k-1} T_i > r \mid (T_1, ..., T_{k-1})  \right)    \right]\\
&= \E \left[ \1_{0 \leq \sum_{i=1}^{k-1} T_i \leq t}  \e^{-\lambda\left(r-\sum_{i=1}^{k-1} T_i\right)}  \right]\\
&= \int_{(\R^+)^{k-1}} \1_{0 \leq \sum_{i=1}^{k-1} s_i \leq t} \; \lambda^{k-1} \e^{-\lambda \sum_{i=1}^{k-1} s_i} \times \e^{-\lambda(r-\sum_{i=1}^{k-1} s_i)} ds_2 ... ds_k\\
&= e^{-\lambda r} \lambda^{k-1} I_{k-1}(t)\\& = e^{-\lambda r} \frac{(\lambda t)^{k-1}}{(k-1)!} \,. 
\end{align*}
with $$I_k(t):= \int_{(\R^+)^k} \1_{0 \leq \sum_{i=1}^k s_i \leq t} ds_1 ... ds_k = \frac{t^k}{k!} \, .$$ Thus  
\begin{align*}
\Proba \left( 0 \leq U^1_{W_1} - U^1_1 \leq t \right)
&= \e^{-\lambda r} + \sum_{k \geq 2} e^{-\lambda r} \frac{\left(\lambda t\right)^{k-1}}{(k-1)!}\\
&= \e^{-\lambda(r - t)}.
\end{align*}
Hence the distribution of $U^1_{W_1} - U^1_1$ is given by $\e^{-\lambda r} \delta_0 + \lambda \e^{- \lambda (r-t)} \1_{(0, r]}(t) dt$.

An easy computation gives $\E(U^1_{W_1} - U^1_1) = r - \frac{1}{\lambda} (1 - \e^{-\lambda r})$. 
Finally we obtain that
\begin{align*}
\E(\tau_1) &= r+A + \E(U^1_1) + \E(U^1_{W_1} - U^1_1)\\
& = r +A + \lambda^{-1}+ r - \lambda^{-1} (1 - \e^{-\lambda r})\\
&=2 r +A + \lambda^{-1} \e^{-\lambda r}
\end{align*}
And from the independence of $U^1_1$ and $U^1_{W_1} - U^1_1$ we conclude that 
\begin{align*}\Var(\tau_1) &= \Var(r+A+U^1_1 + (U^1_{W_1} - U^1_1)\\
&=\Var(U^1_1)+ \Var(U^1_{W_1} - U^1_1)\\
&=\lambda^{-1} + (r^2 - 2 \lambda^{-1} m - m^2)
\end{align*}
From Theorem \ref{thm:lln} we obtain the following LLN  
$$\frac{N^h_t}{t} \cvg[a.s.]{t}{\infty} \frac{1+\lambda r}{2 r + A+ \lambda^{-1} \e^{-\lambda r}} = \frac{\lambda(1+\lambda r)}{\lambda A +1 +(\e^{-\lambda r}-1) +2\lambda r}.
$$

\end{proof}

\section{Proofs.}
\label{sec:proof}
\subsection{Proofs of Proposition \ref{prop:tau_w_h_negative} and \ref{propWexpnegative}}
\label{sec_proof_moment}
We start by proving that the random variables $\tau$ and  $W$ admit exponential moments.
\begin{proof}[Proof of Proposition \ref{prop:tau_w_h_negative}] \quad
\smallskip

Let $h$ be a signed measurable function and $h^+$ its positive part. We generate $N^h$ and $N^{h^+}$ by coupling as in Proposition \ref{prop_EDS}. Recall that $\|h^+\|_1<1$. 

We denote by $W_i$, $\tau_i$, $S_i$, ...  (respectively $W_i^+$, $\tau_i^+$, $S_i^+$, ... )  the renewal quantities associated to $N^h$ (resp. $N^{h^+}$). \emph{Be careful that the previous construction of $W_i^+$, $\tau_i^+$, $S_i^+$ is done by using intervals of length $L(h)$ not $L(h^+)$. Notice that since $L(h)\ge L(h^+)$, then the renewal structure is well defined for $N^{h^+}$. Moreover, if $h\leq 0$, $L(h^+)=0$, one can replace $h^+$ by $h^+_\varepsilon=h^+ + \varepsilon \, \1_{[0,L(h)]}$ and then let $\varepsilon$ go to $0$ in order to compare with \cite{costa}. }

Thanks to Proposition \ref{prop_EDS}, we have $N^h \leq N^{h^+}$ a.s. We also know that $U^{1}_1 = U^{+,1}_1$. 

Moreover, $\tau_1 \leq \tau^{+}_1$ a.s. because the jumps of $N^h$ are included in those of $N^{h^+}$. We also have $W_1 = N^h([0, \tau_1]) \leq N^h([0, \tau^+_1]) \leq N^{h^+} ([0, \tau^+_1]) = W^{+}_1$ a.s. So $W_1 \leq W^+_1 a.s.$
\medskip

\emph{Study of $N^{h^+}$:}~\\
First, we focus on $N^{h^+}$. According to \cite{costa}, we can associate a $M/G/\infty$ queue to $N^{h^+}$. To do this, we consider: 
\begin{align*}
\Lambda^{h^+}(t) = \lambda + \int_{(-L(h), t]} h^+(t-u) N^{h^+}(du).
\end{align*}
We can consider the Hawkes process as the sum of:  
\begin{itemize}
	\item the arrivals of ancestors $V_k$ at rate $\lambda$ and  
	\item a continuous time Galton-Watson process beginning at each $V_k$ with the following characteristics: the number of descendants follows a Poisson distribution with mean $\|h^+\|_1$ and the times of births have the density $h^+/\|h^+\|_1$.
\end{itemize}
In fact, to each arrival of an ancestor $V_k$, we can associate a time $H_k$ corresponding to the life time of the cluster of $V_k$. $V_k$ is independent of $H_k$ and the $(H_k)_k$'s are independent. 

We can associate to this process a queue in the following way:  
\begin{itemize}
	\item the customers are the ancestors and arrive at rate $\lambda$,
	\item the service time for each customer is $H_k + L(h)$.
\end{itemize}
We denote by $Y_t$ the number of customers in the queue at time $t$: 
\begin{align*}
Y_t = \sum_k \1_{V_k \leq t < V_k + H_k + L(h)}.
\end{align*} 
Let $\mathcal{T}^+_1 = \inf\left\lbrace t \geq 0, Y_{t-} \neq 0, Y_t = 0 \right\rbrace$, be the first time the queue is empty. By  proposition 2.6 of \cite{costa}, we have: 
\begin{align*}
\forall	\alpha < \alpha_0:= \min\left(\lambda, \frac{\|h^+\|_1- \log(\|h^+\|_1) - 1}{L(h)} \right), \quad \textrm{ it holds } \quad \E[\e^{\alpha \mathcal{T}^+_1}] < \infty \, .
\end{align*}
Of course $\lambda >0$ and $\|h^+\|_1- \log(\|h^+\|_1) - 1 > 0$, and so $\mathcal{T}^+_1$ admits an exponential moment. 

Since $\tau^+_1$ is the first time after $U^{+, 1}_1$ such that there were no jump during a time $L(h)$. Thus $\tau^+_1 = \mathcal{T}^+_1$ and since $\tau_1 \leq \tau_1^+$, part (i) of the proposition is proved.
\medskip

In order to prove (ii) it is enough to show that the distribution of $W_1^+$ admits exponential moments. Recall that $$W_1^+ = N^{h^+}([0,\tau_1^+]) \, .$$ According to \cite{bordenave} (see proof of Theorem 3.2 and proof of Theorem 3.4 therein), $$\lim_{t \to +\infty} \, \frac 1t \, \ln \mathbb E\left(e^{\theta N^{h^+}([0,t])}\right) = \lambda(\mathbb E(e^{\theta \mathcal{S}}) -1) :=\mu(\theta) < +\infty$$ as soon as $\theta < \|h^+\|_1- \log(\|h^+\|_1) - 1$. Here $\mathcal{S}$ is distributed according to (see (3) in \cite{bordenave}) $$\mathbb P(\mathcal{S}=k) = \frac{e^{-k\|h^+\|_1}(k\|h^+\|_1)^{k-1}}{k!} \, .$$ It is thus immediate that $\mu(\theta)$ goes to $0$ as $\theta$ goes to $0$. 

For $\varepsilon >0$ we may thus choose $\theta$ small enough such that $$\alpha_0 - 2\varepsilon \geq \mu(2 \theta) + \varepsilon \, .$$ For this $\theta$, one can find $t_\theta$ such that for $t \geq t_\theta$, $$\mathbb E\left(e^{2 \theta N^{h^+}([0,t])}\right) \leq e^{t (\mu(2\theta) +\varepsilon)} \, .$$ It follows
\begin{eqnarray*}
\mathbb E(e^{\theta W_1^+}) &=& \mathbb E\left(e^{\theta N^{h^+}([0,\tau_1^+])}\right)\\  &\leq& \sum_{k=1}^{\infty} \, \mathbb E\left(e^{\theta N^{h^+}([0,k])} \, \1_{k-1\leq \tau_1^+ < k}\right) \\ &\leq& \sum_{k=1}^{\infty} \, \left( \beta_k \, \mathbb E\left(e^{2\theta N^{h^+}([0,k])}\right) \, + \, \frac{1}{\beta_k} \mathbb P(k-1 \leq \tau_1^+) \right) \\ &\leq&  A(t_\theta) \, + \, \sum_{k=[t_\theta]+1}^{\infty} \, \left(\beta_k \, e^{k(\mu(2\theta)+\varepsilon)} \, + \, \frac{\mathbb E(e^{(\alpha_0-\varepsilon)\tau_1^+})}{\beta_k} \, e^{- (k-1) (\alpha_0 - \varepsilon)}\right)
\end{eqnarray*}
where $A(t_\theta)$ denotes the finite sum up to $k=[t_\theta]$. Choosing $\beta_k= k^{-2} \, e^{-k(\mu(2\theta)+\varepsilon)}$ the $k$'th term of the remaining sum is smaller than $1/k^2 + c \, k^2 \, e^{- \varepsilon (k-1)}$ and the series is thus convergent. Since $\varepsilon$ is arbitrary, (ii) follows. 
\end{proof}
\medskip

\begin{proof}[Proof of Proposition \ref{propWexpnegative}] \quad
\smallskip

We consider a process $N^h$, generated by the Poisson point process $Q$, as in the Proposition \ref{prop_EDS}. Since $h\leq 0$, we will couple (and upper-bound) this time the Hawkes process with the Poisson point process $\mathcal{R}$ on $\R^+$, with intensity $\lambda$, generated by the same Poisson point process $Q$ on $(0,\infty)^2$. Since $\forall t\ge0$ $$\lambda\ge \Lambda^h(t)\quad \text{a.s.}$$ we deduce that $$\mathcal{R}\ge N^h.$$
We can now upper bound the length of the first time window $\tau_1$ by a similar quantity associated with $\mathcal{R}$. Recall that $U^1_1$ is the first jump time of $N^h$ and define:
\begin{equation}\label{def:tau}\tau = \inf \{t>U_1^1, \mathcal{R}[t-L(h),t) \neq 0, \mathcal{R}(t-L(h),t] =0 \}.
\end{equation}
$\tau$ indicates the first moment such that there were no jump for the process $\mathcal{R}$ during an interval of length $L(h)$. In particular, there weren't jump for $N^h$ either. Therefore  $\tau_1 \leq \tau$ a.s. and $$W_1 = N^h([U^1_1, \tau_1]) \leq \mathcal{R}([U^1_1, \tau_1]) \leq \mathcal{R}([0, \tau]).$$ It is thus enough to get an upper bound for $\E\left(e^{\theta \, \mathcal{R}([0, \tau])}\right)$.

To this end we shall study the random variable $\tau$. Denote by $V_i$ the jumps of the Poisson point process $\mathcal{R}$. From the definition there exists a random integer $K$ such that 
$$\tau = V_K + L(h).$$ 
The definition of $K$ leads to $$K=\mathcal{R}[0,\tau].$$
From the independence of the times between jumps of $\mathcal{R}$ we deduce that 
\begin{align*}
\Proba(K=1)=\Proba[\tau = V_1 +L(h)] &= \Proba[V_2 - V_1 \geq L(h)]\\
	&= \e^{- \lambda L(h)},\\
\Proba(K=2)=\Proba[\tau = V_2 + L(h)] &= \Proba[\{V_2- V_1 < L(h)\} \cup \{V_3-V_2 \geq L(h)\}]\\
	&= \Proba[V_2- V_1 < L(h)]\Proba[V_3-V_2 \geq L(h)]\\
	&= (1 - \e^{- \lambda L(h)}) \e^{- \lambda L(h)},\\
\forall	k \geq 2,\quad \Proba(K=k)= \Proba[\tau = V_k+L(h)] &= (1-\e^{- \lambda L(h)})^{k-1} \e^{- \lambda L(h)}.
\end{align*}
$K$ is a geometric random variable with parameter $e^{-\lambda L(h)}$ and thus admits exponential moments provided $e^\theta \, (1-e^{-\lambda L(h)}) < 1$ which concludes the proof.
\end{proof}
\bigskip

\subsection{Proof of the LLN and CLT}

\begin{proof}[Proof of Theorem \ref{thm:lln} and Theorem \ref{thm:clt}] \quad
\smallskip

Recall that $$\hat{N}^h_t \leq N^h_t \leq \hat{N}^h_t + W_{M^h_t+1} \quad \text{a.s.}$$ where $$\hat{N}^h_t := \sum_{i =1}^{\infty} W_i \1_{S_{i} \leq t} = \sum_{i=1}^{M^h_t} W_i$$ and $$M_t^h := \, \sum_{i=1}^{\infty} \, \1_{S_i\leq t} \, ,$$ as explained in \eqref{eqcumul1}, \eqref{eqcumul2}, \eqref{eqcumul3}.

As we previously said Theorem 3.1 and Theorem 3.2 Chapter 6 in \cite{asmussen}  furnish $$ \frac{\hat{N}^h_t}{t} \cvg[a.s.]{t}{\infty} \frac{\E[W_1]}{\mathbb E(\tau_1)} \, $$ and $$\sqrt t \, \left(\frac{\hat{N}^h_t}{t} \, -  \, \frac{\E[W_1]}{\mathbb E(\tau_1)}\right) \underset{t\to\infty}{\Longrightarrow} \mathcal{N}(0,\sigma^2)$$with $$\sigma^2 = \frac{\Var\left(W_1-\tau_1 \, \frac{\E[W_1]}{\mathbb E(\tau_1)}\right) }{\mathbb E(\tau_1)} \, .$$ 
It is thus enough to control the remaining (or error) term $W_{M^h_t+1}$ i.e to prove
\begin{equation}
\label{eq:lln_reste}
\lim_{t\to\infty} \frac{W_{M^h_t+1}}{t} =0 \quad \text{a.s.},
\end{equation}
and 
\begin{equation}\label{eqrestproba}
\lim_{t\to\infty} \frac{W_{M^h_t+1}}{\sqrt t} =0 \quad \text{in Probability}.
\end{equation}
\eqref{eq:lln_reste} will conclude the LLN and \eqref{eqrestproba} the CLT.
\medskip

Actually we will prove stronger results. Let $\beta(n)$ an increasing sequence going to infinity and 
$\varepsilon >0$. Introduce the independent events $A_n = \{W_n > \varepsilon \beta(n) \}$. Then $\limsup_n A_n = \{\limsup_n \frac{W_n}{\beta(n)} > \varepsilon \}$. Since the $(W_i)_i$ are finite i.i.d random variables
$$\sum_{n} \Proba(A_n) = \sum_{n} \Proba \left(W_n > \varepsilon \beta(n)\right) = \sum_n \Proba (W_1 > \varepsilon \beta(n) ) \, .$$ Thanks to Proposition \ref{prop:tau_w_h_negative} and to Markov inequality, we know that for $$\mathbb P(W_1 > \varepsilon \beta(n)) \leq \, \mathbb E[e^{\theta_0 W_1}] \, e^{-\theta_0 \, \varepsilon \, \beta(n)} \, .$$ We may now apply Borel-Cantelli, telling that provided $\sum_n \, e^{-\theta_0 \, \varepsilon \, \beta(n)} < +\infty$, $$\Proba(\limsup_n A_n) = 0.$$ The previous holds with $\beta(n)=n^\alpha$ for any $\alpha >0$. We have proved in particular that $$\frac{W_n}{\sqrt n} \cvg[a.s.]{n}{\infty} 0 \, .$$
Since $M^h_t$ is a non-decreasing family of integers going to infinity almost surely, $$\frac{W_{M^h_t+1}}{\sqrt{M^h_t+1}} \cvg[a.s.]{t}{\infty} 0 \, .$$ It remains to recall that 
\begin{equation}\label{eq:lim2}
\frac{M^h_t +1}{t} \cvg{t}{\infty} \frac{1}{\mathbb E(\tau_1)} \, \textrm{ a.s.},
\end{equation}
to conclude that $$\frac{W_{M^h_t+1}}{\sqrt{t}} \cvg[a.s.]{t}{\infty} 0 \, .$$
\end{proof}
\bigskip
\subsection{Proof of Corollary \ref{corLD}}
\begin{proof}[Proof of Corollary \ref{corLD}] \quad

In order to prove  the first part of Corollary \ref{corLD} it is enough to show that  $\hat N^h_t/t$ and $N^h_t/t$ are exponentially equivalent, i.e. that for each $\delta >0$, $$\limsup_{t \to \infty} \, \frac 1t \, \ln \mathbb P\left(\left|\frac{N_t^h}{t} \, - \, \frac{\hat N_t^h}{t}\right| > \delta\right) \, = \, - \, \infty \, .$$ To this end it is enough to show that 
\begin{equation}\label{eqldequiv}
\limsup_{t \to \infty} \, \frac 1t \, \ln \mathbb P\left(W_{M_t^h+1} > \delta t\right) = - \, \infty \, .
\end{equation}

We will decompose the probability state into two events: $M^h_t \leq t^2$ and $M^h_t >t^2$. It holds 
	\begin{align*}
	\Proba \left(W_{M^h_t+1} > \delta \, t \right) &\leq \Proba \left( M^h_t > t^2 \right) + \Proba \left( \left\{ W_{M^h_t+1} > \delta t \right\} \cap \left\{ M^h_t \leq t^2  \right\}  \right)\\
	&\leq \Proba \left( M^h_t > t^2 \right) + \Proba \left( \left\{ \exists k \in \{1, ..., \lfloor t^2 +1 \rfloor \}, W_{k} > \delta t \right\} \cap \left\{ M^h_t \leq t^2 +1  \right\} \right)\\
	&\leq \Proba \left( M^h_t > t^2 \right) + \Proba \left(\exists k \in \{1, ..., \lfloor t^2 +1 \rfloor \}, W_{k} > \delta t \right)\\
	&\leq \Proba \left( M^h_t > t^2 \right) + \sum_{j=1}^{\lfloor t^2 +1 \rfloor}\Proba \left(W_{k} > \delta t \right)\\
	&\leq \Proba \left( M^h_t > t^2 \right) + (t^2+1) \Proba \left(W_1 > \delta t \right). 
	\end{align*}
	
	On one hand, we have, by Markov's inequality, for all $\theta_0 >\theta >0$, $$\Proba\left(  W_{1} > \delta t \right) \leq \E[\e^{\theta W_1}] \, \e^{- \theta \delta t}$$ so that $$\limsup _{t \rightarrow + \infty} \frac{1}{t} \ln \Proba\left(  W_{1} > \delta t \right) \leq - \theta_0 \delta. $$

On the other hand, according to \cite{tiefeng} Theorem 2.3, for all $x >0$: 
	\begin{align*}
	\varlimsup_{t \rightarrow + \infty} \frac{1}{t} \ln \Proba \left(\frac{M^h_t}{t} \geq x \right) \leq - J_{\tau_1}(x), 
	\end{align*}
where $J_{\tau_1}(x) = \sup_{\eta} \{\eta - x \ln \E[\e^{\eta \tau_1}]\}$. \\ Since $\eta \mapsto \mathbb E(e^{\eta \tau_1})$ is continuous on $\mathbb R^-$ there exists some $\eta_0$ such that $\mathbb E(e^{\eta_0 \tau_1})=e^{-1}$. It follows $J_{\tau_1}(x) \geq \eta_0 +x$.\\
Choose $t_1, t_2, ...$ an increasing sequence of times such that $t_i \cvg{i}{+\infty} + \infty$. For a fixed $i$, we have for $t$ large enough
	\begin{align*}
	\Proba \left( M^h_t > t^2 \right) = \Proba \left(\frac{M^h_t}{t} > t\right) \leq \Proba \left( \frac{M^h_t}{t} > t_i \right)\\
	\end{align*}
Since 
	\begin{align*}
	\limsup_{t \rightarrow + \infty} \frac{1}{t} \ln \Proba \left( \frac{M^h_t}{t} > t_i \right) \leq - J_{\tau}(t_i) \leq - t_i - \eta_0. 
	\end{align*}
It follows, 
	\begin{align*}
	\limsup_{t \rightarrow + \infty} \frac{1}{t} \ln \Proba \left( \frac{M^h_t}{t} > t \right) = - \infty. 
	\end{align*}
Eventually, 
	\begin{align*}
	\limsup_{t \rightarrow + \infty}& \frac{1}{t} \ln \Proba \left( \frac{W_{M^h_t+1}}{t} > \delta  \right) \\
	&\leq \limsup_{t \rightarrow + \infty} \frac{1}{t} \ln \left[ \Proba \left( M^h_t > t^2 \right) + (t^2+1) \Proba \left(W_1 > \delta t \right)\right]\\
	&\leq \limsup_{t \rightarrow + \infty} \left( \frac{\ln 2}{t} + \max \left[ \frac{1}{t} \ln \Proba \left( \frac{M^h_t}{t} > t \right), \frac{1}{t} \ln \left((t^2+1) \Proba \left(W_1 > \delta t \right)\right)   \right] \right)\\
	&\leq \max \left[ \limsup_{t \rightarrow + \infty} \frac{1}{t} \ln \Proba \left( \frac{M^h_t}{t} > t \right), \limsup_{t \rightarrow + \infty} \left( \frac{ \ln (t^2+1)}{t} + \frac{1}{t} \ln \Proba \left(W_1 > \delta t \right) \right)  \right]\\
	& \leq - \theta_0 \, \delta \, .
	\end{align*}
This completes the proof for $\theta_0=+ \infty$. 
\medskip

Let us now assume $\theta_0<\infty$. Recall that $m= \frac{\mathbb E(W_1)}{\mathbb E(\tau_1)}$, then \eqref{eqaudessus} is a consequence of the following line of reasoning: 
\begin{eqnarray*}
\mathbb P\left(\frac{N^h_t}{t} \,  > \, m+ a \right) &\leq& 
\mathbb P\left(\frac{\hat N^h_t}{t} \, + \frac{W_{M_t^h+1}}{t} > \, m+ a \right) \\ &\leq& \mathbb P\left(\frac{\hat N^h_t}{t} \, > \, m+ \kappa_1 a \right) +  \mathbb P\left(\frac{W_{M_t^h+1}}{t} > (1-\kappa_1) a \right) 
\end{eqnarray*}
where $\kappa_1\in(0,1)$, yielding 
\begin{align*}
\limsup_t &\, \frac 1t \, \ln\mathbb P\left(\frac{N^h_t}{t} \,  > \, m+ a \right) \\
&\leq \max\left( \, \limsup_t \, \frac 1t \, \ln \mathbb P\left(\frac{\hat N^h_t}{t} \,  > \, m+ \kappa_1 a \right) \, , \, \limsup_t \, \frac 1t \, \ln \mathbb P\left(\frac{W_{M_t^h+1}}{t} > (1-\kappa_1) a \right) \right)
\end{align*}Now applying Theorem \ref{thmLD} with $\kappa_2$ and $(1-\kappa_2)$ instead of $1/2$, we deduce that
\begin{align*}
\limsup_t &\, \frac 1t \, \ln\mathbb P\left(\frac{N^h_t}{t} \,  > \, m+ a \right) \\
&\leq \max \left( - \inf_{z-m \geq   \kappa_2 \kappa_1 a} \, J(z) \, , \, - (1-\kappa_2) \kappa_1 \theta_0  a\,,\,  - (1 - \kappa_1) a \theta_0  \right)
\end{align*}
 yielding the result with $\kappa = \kappa_1 \kappa_2$ and $\kappa' = 1 - \kappa_1$. The condition $\kappa + 2 \kappa' =1$ arises from the equality of the last two terms

Finally, \eqref{eqendessous} is a consequence of the same reasoning on $\hat{N}^h_t \leq N^h_t$: 
\begin{eqnarray*}
	\mathbb P\left(\frac{N^h_t}{t} \,  < \, m- a \right) &\leq& 
	\mathbb P\left(\frac{\hat N^h_t}{t} \, < \, m - a \right) 
\end{eqnarray*}
 yielding the result.
\end{proof}
\medskip

\section{Final comments.}

As we recalled in the introduction, in the linear case the LLN, the CLT and the LDP are completely characterized by $||h||_1$. As we have shown in Section \ref{secexample}, the ``almost linear case with inhibition'' we are looking at is dramatically different, since the results do not only depend even on the moments of $h$. The renewal description of the Hawkes process we have used allows us to characterize all these limit theorems in terms of the joint law of $(\tau_1,W_1)$. It should be very interesting to link this distribution with $h$. As for the non linear self-excited case such a goal seems difficult to reach.

Another interesting direction should be to obtain non asymptotic deviation bounds (or concentration properties). Since the Large Deviation Principle for cumulative processes we have proved in \cite{CCC2} is based on the contraction of a higher level LDP, new methods are necessary for non asymptotic results.

The methods of the paper can be used for more general jump rate functions $f$, provided one can generalize the construction of the sequence $(\tau_i,W_i)$. This generalization is partly done in \cite{graham} in which a regenerative structure is exhibited without the assumption of bounded support for the reproduction function $h$ and in \cite{raads} which exhibit renewal points for non linear Hawkes processes and age-dependent Hawkes processes.

\end{document}